\documentclass{amsart}
\usepackage{amsmath,amssymb,hyperref,mathrsfs,graphicx}
\usepackage[utf8x]{inputenc}
\usepackage{amsmath}
\usepackage{amsfonts}
\usepackage{latexsym}
\usepackage{amsthm}
\usepackage{amssymb,amscd}
\usepackage{xargs}
\usepackage{tikz}
\usepackage{graphicx}
\usepackage{color} 
\usepackage{enumitem}

\newtheorem{thm}{Theorem}[section]

\newtheorem{lem}[thm]{Lemma}
\newtheorem{defn}[thm]{Definition}

\newtheorem{cor}[thm]{Corollary}
\newtheorem{rmk}[thm]{Remark}
\newcommand{\be}{\begin{eqnarray}}
\newcommand{\ee}{\end{eqnarray}}
\newcommand{\bee}{\begin{eqnarray*}}
\newcommand{\eee}{\end{eqnarray*}}
\newcommand{\beal}{\begin{aligned}}
\newcommand{\enal}{\end{aligned}}

\newcommand{\eps}{\varepsilon}

\newcommand{\T}{\mathbb{T}}
\newcommand{\R}{\mathbb{R}}
\newcommand{\Q}{\mathbb{Q}}
\newcommand{\cG}{\mathcal{G}}

\newcommand{\N}{\mathbb{N}}
\newcommand{\Z}{\mathbb{Z}}

\newcommand{\Om}{\Omega}

\newcommand{\cE}{\mathcal{E}}

\newcommand{\cZ}{\mathcal{Z}}

\newcommand{\cN}{\mathcal{N}}

	\def\textr{\textcolor{red}}

\title{coexistence of $1/2,1/3-$caustics for deformative nearly circular billiard maps}
\author{Vadim Kaloshin\dag}
\address{\dag\  University of Maryland, College Park, MD, USA }
\email{vadim.kaloshin@gmail.com}
\author{Jianlu Zhang\ddag}
\address{\ddag\ Institute of Theoretical Studies, ETH Z\"urich, CH-8092 Z\"urich, Switzerland}
\email{jianlu.zhang@math.ethz.ch}

\thanks{}
\date{}
\begin{document}
\maketitle
\begin{abstract} For symmetrically analytic deformation of the circle (with certain Fourier decaying rate), the necessary condition for the corresponding billiard map to keep the coexistence of $1/2,1/3$ caustics is that the deformation has to be an isometric transformation.
\end{abstract}
\vspace{10pt}

\section{Introduction}\label{0}\vspace{10pt}
Suppose $\Om\subset \R^2$ is a strictly convex domain, with the boundary $\partial\Om$ is $C^r$ smooth, $r\geq 2$. The billiard problem inside $\Om$ can be described as the following:  \\

{\it A massless particle moves with unit speed and no friction following a rectilinear path inside the domain $\Om$. When the ball hits the boundary, it is reflected elastic-
ally according to the law of optical reflection: the angle of reflection equals the angle of incidence.} \\

This problem was first investigated by Birkhoff (see [3]). Later we can see that such trajectories are called {\bf broken geodesics}, as they correspond to local maximizers of the distance functional. The {\bf billiard map} can be identified by the correspondence of the positions in one reflection $\phi:P_0\rightarrow P_1$, see Fig. \ref{fig1}.\\

\begin{figure}
\begin{center}
\includegraphics[width=6cm]{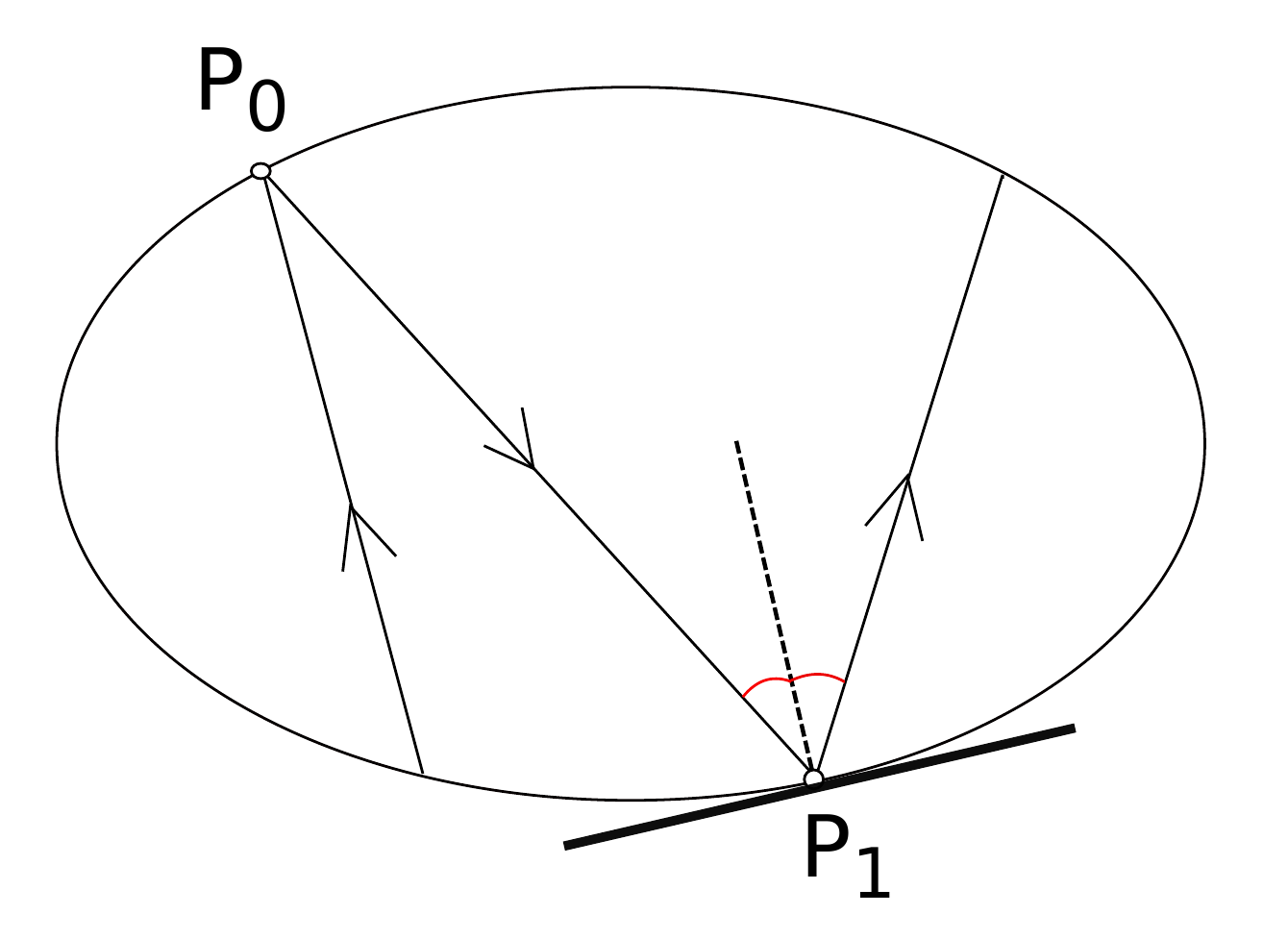}
\caption{The reflective angle keeps equal to the incident angle for every rebound.}
\label{fig1}
\end{center}
\end{figure}

Now let's introduce a coordinate for the billiard map. Suppose that $\partial\Om$ is parametrized by arc length $s$ with the circumference of $\partial\Om$ equal to $1$, and $v$ is the angle between $P_1-P_0$ and the positive tangent to $\partial\Om$ at $P_0$. Then the billiard map can be described as:
\be
\phi:(s,v)\rightarrow(s_1,v_1)
\ee
defined on the closed annulus $\mathbb{A}=[0,1]\times[0,\pi]$. Obviously $\phi|_{\partial\mathbb{A}}=Id$, i.e. 
\[
\phi(s,0)=(s,0),\quad\phi(s,\pi)=(s,\pi),\quad\forall s\in\T.
\]
Let's denote by 
\be
h(x,x'):=-d(P_0,P_1), \quad (x,x')\in\R^2, 
\ee
where $d(\cdot,\cdot)$ is the Euclid distance of $\R^2$, $x\equiv s\;(mod \;2\pi)$ and $x'\equiv s'\;(mod \;2\pi)$. It's easily to find that 
\[
h(x+1,x'+1)=h(x,x'),\quad(x,x')\in\R^2
\]
and
\be\label{p-h}
\partial_1 h=\cos v,\quad \partial_2h=-\cos v'.
\ee
The twist property is implied by $-\partial_{12}h>0$ once the boundary is strictly convex \cite{M2}. \\

Let's make a rule for $s\in\T=\R/[0,2\pi]$: we always choose the counter clockwise direction to order the configurations, that means for any $q-$tuple $(s_0,s_1,\cdots,s_{q-1})$ with $s_i\in\T$, $q\geq2$ and $0\leq i\leq q-1$, we can fix a unique configuration $(x_0,x_1,\cdots,x_{q-1})$ in the universal covering space $\R$, such that 
\[
x_i\equiv s_i\;(mod\; 2\pi)\text{\;for\;} 0\leq i\leq q-1,
\]
\[
 \quad  0\leq x_{i+1}-x_i<1\text{\;for\;} 0\leq i\leq q-2,
 \]
 \[
 \quad x_0\in[0,1).
\]
This unique configuration implies the following definition:
\begin{defn}
For a $q-$periodic tuple $S=(s_0,s_1,\cdots,s_{q-1},s_q)$ with $s_0=s_q$, we define the {\bf winding number} of it by 
\[
p:=\frac{x_q-x_0}{2\pi}\in\Z^+
\]
and the {\bf rotation number} by
\[
\rho(S)=\frac p{ q},
\]
where $X=(x_0\cdots,x_q)$ is the configuration corresponding to $S$ defined above.
\end{defn}
\begin{rmk}
Due to previous setting, we can see that for any $q-$periodic tuple $S$, the rotation number $\rho(S)\in[0,1)\cap\Q$.
\end{rmk}

For $q\geq2$ and a fixed $s\in\T$, we can define a $q-$periodic configuration set $C_{p/q}(s)$ by
\[
C_{p/q}(s):=\big\{(x_0,\dots,x_q)\in\R^{q+1}\big|0\leq x_{i+1}-x_i<1\text{\;for\;} 0\leq i\leq q-2,x_q-x_0=p,x_0\equiv s\;(mod\; 2\pi)\big\}.
\]
Then the following {\bf action function} is well defined
\be\label{per-action}
F_{p/q}(s):=\inf_{\gamma\in C_{p/q}(s)}\sum_{i=0}^{q-1}h(\gamma_i,\gamma_{i+1})
\ee
and the minimizer $\gamma^*$ obeys the discrete Euler Lagrange equation 
\be\label{e-l}
\partial_1h(\gamma^*_i,\gamma^*_{i+1})+\partial_2h(\gamma^*_{i-1},\gamma^*_i)=0,\quad\forall\; i=1,2,\dots,q-1.
\ee
Moreover, due to (\ref{p-h}) the corresponding $\{(s_i,v_i)\}_{i=0}^{q}$ is an orbit of the billiard map $\phi$, see \cite{B}.
\begin{rmk}
Notice that $\{(s_i,v_i)\}_{i=0}^{q}$ may not be a `real' periodic orbit, because $v_0\neq v_q$ could happen. However, if we interpret $F_{p/q}(s)$ as a function defined on $\T$, then the minimizer $s^*$ will definitely correspond to a `real' periodic orbit $\gamma^*$ of the billiard map.
\end{rmk}

\begin{defn} We call a (possibly not connected) curve $\Gamma\subset\Om$ a {\bf caustic} if any billiard orbit having one segment tangent to $\Gamma$ has all its segments tangent to it. Precisely, $\Gamma_{p/q}$ is an {\bf $p/q-$rational caustic} if all the corresponding (noncontractible) tangential orbits are periodic with the rotation number $p/q$.
\end{defn}

\begin{rmk}
In the remaining part of this paper, we agree that all caustics that we will consider will be smooth and convex; we will refer to such curves simply as caustics. By the Birkhoff's Theorem of general exact monotone twist maps, such a rational caustic will correspond to a $\phi-$invariant curve in the phase space, i.e. formally we can express by 
\[
\Gamma_{p/q}=\{(s,g_{p/q}(s))\in\mathbb A|s\in\T\}
\]
which consists of periodic orbits of ratation number $p/q$. Moreover, $\Gamma_{p/q}$ is non-contractible and then $g_{p/q}(s)$ is a Lipschitz graph, see \cite{B2}.

Besides, there is reversibility of the billiard map $\phi$, i.e. $\phi\circ I(s,v)=I\circ\phi^{-1}(s,v)$ for all $(s,v)\in\mathbb A$, where $I(s,v)=(s,\pi-v)$ is a reflective transformation. Benefit from this, we can see that $\Gamma_{p/q}=\Gamma_{1-p/q}$. So in the following we can impose that the rotation number of the caustic belongs to $[0,1/2]$.
\end{rmk}
\begin{thm}\cite{M}
For an exact symplectic twist map, every orbit on a non contractible invariant curve is minimizing in the sense of variation.
\end{thm}
By applying this Theorem to the billiard map, we get
\begin{cor}
Suppose the billiard map $\phi$ has a $p/q-$rational caustic, then $F_{p/q}(s)$ equals a constant for all $s\in\T$.
\end{cor}
\begin{rmk}


In \cite{GK}, the authors defined the {\bf width function $w(\alpha)$} of a convex billiard curve by the strip width formed by the tangent rays with the direction $\alpha+\pi/2$ and $\alpha-\pi/2$, for any unit vector $\alpha$ in $\R^2$. We called the billiard boundary {\bf of constant width}, if $w(\alpha)$ is a constant. Observe that billiard boundaries of constant width has $1/2$ caustic.

The circle is a trivial example of constant width, of which the $1/2$ caustic is the centre point. A popular non trivial example is the Releaux triangle, see \cite{YB}. 

We should remind the readers that elliptic boundary has no $1/2$ caustic. Indeed, the elliptic billiard has a first integral, see \cite{PR}, but the $1/2$ periodic orbits just correspond to the rebound along the major (reps. minor) axis, which have separatrix arcs containing homoclinic orbits but not periodic ones.
 \end{rmk}
 \vspace{10pt}
\noindent{\bf Organization of this article.} In Section \ref{1} we state the rigid constraints on the boundaries to preserve certain rational caustics; we introduce the Projection Theorem and our main conclusion towards it; In Section \ref{2} we get the $1^{st}$ and $2^{nd}$ order estimate of the action function, which leads to our harmonic equation; In Section \ref{3} we analysis the harmonic equation and give the proof of our main conclusion. In Section \ref{4} we make some heuristic comments and further generalizations of this direction. \\

\section{\sc Main conclusion and Scheme of the proof}\label{1}\vspace{10pt}

Based on previous section's setting, a natural question can be asked: How far can we constraint the boundary to preserve certain rational caustics?\\

Recently, Avila, de Simoi and Kaloshin proved the following result:
\begin{thm}\cite{ADK}
There exist $e_0 > 0$ and $\eps_0 > 0$ such that for any $0 \leq e \leq e_0$, $0 \leq \eps < \eps_0$, any rationally integrable $C^{39}-$smooth domain $\Om$ so that $\partial\Om$ is $C^{39}$ $\eps-$close to the ellipse $\cE_e$ with the eccentricity $e$ is also an ellipse.
\end{thm}
This implies that locally ellipse with small eccentricity is the only possible integrable convex billiard boundaries. The proof of this Theorem relies on an improved $1^{st}$ order estimate of the action function. To help the readers to get a clearer understanding of that, let us start by exploring the integrable infinitesimal deformations of a circle. Suppose $\Om_0$ be the unit disk, and the polar coordinates on the plane be $(r, \theta)$. Let $\Om_\eps$ be a one-parameter family of deformations given by 
\[
\partial\Omega_\eps=\{r  = 1 + \eps  n(\theta) +O(\eps^2) \},
\]
where the Fourier expansion of $n$:
\[
n(\theta)=n_0+\sum_{k>0}n'_k\sin k\theta+n''_k\cos k\theta.
\]
The following Theorem is available:
\begin{thm}[Ramirez-Ros \cite{R}]\label{r-r}
If $\Om_\eps$ has an integrable rational caustic $\Gamma_{1/q}$ of rotation number $1/q$ for all sufficiently small $\eps$, then $n'_{kq}=n''_{kq}=0$ for all $k\in\N$.
\end{thm}
If we impose that $\Om_\eps$ is rationally integrable for all $1/q$ with $q>2$ and sufficiently small $\eps$, then the above theorem implies that $n'_k=n''_k=0$ for all $k>2$, i.e. we establish $n(\theta)$ by:
\be\label{ellip}
n(\theta)&=&n_0 +n'_1\cos\theta+n''_1\sin\theta+n'_2\cos2\theta+n''_2\sin2\theta\nonumber\\
&=&n_0 +n^*_1\cos(\theta-\theta_1)+n^*_2\cos(\theta-\theta_2).
\ee
\begin{rmk}
As in \cite{ADK}, we can find that in (\ref{ellip}) $n_0$ corresponds to a homothety, $n_1^*$ corresponds to a shift in the direction $\theta_1$ and $n_2^*$ corresponds to a deformation into an ellipse of small eccentricity with the major axis coincides with the direction $\theta_2$.

It's remarkable that in the above theorem, one have to face $\eps\rightarrow 0$ as $q\rightarrow\infty$. So they need to replace the infinitesimal deformation by a fixed lower bound such that $\eps>\eps_0$; You can see \cite{ADK} for more technical details.
\end{rmk}

A natural generalization of previous Theorem is reducing to `finitely many' rational caustics, based on which we can still get the integrability of the billiard maps:\\

\noindent{\bf The Projected Conjecture.}
In a $C^r$ ($r=2,\cdots,\infty,w$) neighborhood of the circle there is no other billiard domain
of constant width and preserving $1/3$ caustics.\\

\begin{defn}
Denote by $BZ_q$ the manifold of codimension
infinity containing all the strictly convex billiard boundaries preserving the $1/q$ caustic $\Gamma_q$, $q\geq 2$.
\end{defn}
Here we propose a possible approach to prove it, still we start from the infinitesimal deformation of the circle:\\

{\bf Step 1.} Find the tangent bundle of $BZ_q$ at the circle. Let $n(\theta)$ and $m(\theta)$ be functions with $\theta\in\T^1=\R/[0,2\pi]$ given by the Fourier series
\[
n(\theta) = \sum_{k\in \Z} n_k \exp i(k\theta) \qquad , \qquad 
m(\theta) = \sum_{k\in \Z} m_k \exp i(k\theta).
\]
In the polar angle coordinate, for sufficiently small $\eps$ the perturbed domain can be denoted by
\[
\partial\Omega_\eps=\{r  = 1 + \eps  n(\theta) +\eps^2 m(\theta)+O(\eps^3) \} \qquad (*)
\] 
%

Denote by $T_q$ the tangent space of $BZ_q$ at the circle, and $T_q^\perp$ be the orthogonal complementary of $T_q$. Due to Theorem \ref{r-r}, $T_q$ consists of functions given by Fourier coefficients whose indices are not divisible by $q$. \\

{\bf Step 2.} Describe $BZ_q$ as a graph over $T_q$. Namely, $BZ_q = \{ (n,F_q(n)) \in T_q \times T^\perp_q\}$.

\

{\bf Step 3.} Show that for all
$n \in T_2 \cap T_3\setminus 0$, there exists a upper bound $\eps_0=\eps_0(n)$\, such that for all $0<\eps\leq\eps_0$, we have $F_2(\eps n)\ne F_3(\eps n)$.
In particular, it implies the Projected Conjecture.

\begin{rmk}
Notice that $\eps_0(n)$ may be not uniform for different $n\in C^r(\T,\R)$, even though $\|n\|_{C^r}=1$ is imposed. So we need a similar approach as \cite{ADK} to avoid the the collapse of the infinitesimal $\eps_0(\cdot)$. To keep the consistency and readability, we don't consider this part in this paper.
\end{rmk}
\vspace{10pt}

Following aforementioned strategy, we claim our main conclusion. Before we doing that, let's formalize the symbol system first:\\
\begin{itemize}
\item $\vec{\gamma}_0\in\R^2$ be the unit circle and $\vec{\gamma}_\eps\in\R^2$ be a deformation, and in the polar coordinate $r_0$, $r_\eps$ represent the corresponding axis length. \\
\item $\{z_k^0\in\Omega_0\}$, $\{z_k^\eps\in\Omega_\eps\}$ be the configuration of $q-$periodic, $k=0,1,\cdots,q-1$.\\
\item $s\in[0,2\pi]$ be the arc length  variable and $\theta\in[0,2\pi]$ be the rotational angle; \\
\item Suppose $\gamma_t(\theta)$ is a curve of strictly convex boundaries in $\R^2$ with the parameter $t\in[0,1]$ and starting from the circle, i.e. $\gamma_0(\theta)=1$. If $\gamma_t(\theta)$ is $C^3$ smooth of $t$, then we can expand the curve in the polar coordinate by
\be
r_t(\theta)=1+t n(\theta)+ t^2 m(\theta)+O(t^3),\quad t\ll1.
\ee
We can assume $n,m\in C^w(\T,\R,\rho)$, $\rho>0$. Moreover, by rescaling $t$, we can make
\be\label{rescale}
|n(\theta)|_{C^0}=1,\;  |m(\theta)|_{C^0}\leq C.
\ee
\item Recall that two different deformation $\gamma_t$ and $\gamma'_t$ may be homogenous by a rigid transformation on $\R^2$ space (parallel shift, rotation), so we just need to choose a representation by imposing
\be\label{rigid}
n(0)=0, n'(0)=0.
\ee
\item Besides, we can fix the perimeter by $2\pi$ for all the deformations, i.e.
\be\label{flux}
\int_\T n(\theta)d\theta=0.
\ee
\end{itemize}
\begin{thm}[Main Conclusion]
For non degenerate deformation of circle which can be formalized by $(*)$ and satisfying (\ref{rescale}), (\ref{rigid}) and (\ref{flux}), if $n(\theta)$ is an even analytic function with super exponential decaying rate, i.e.
\[
\lim_{i\rightarrow\infty}\frac{w(i)}{2^i}=\infty
\]
with the modular function $w:\Z\rightarrow \R^+$ satisfying 
\[
w(-i)=w(i),
\]
\[
n(\theta)=\sum_{k\in\Z}n_k e^{ik\theta}, \quad |n_k|\leq e^{-w(k)}
\]
and there exists a subsequence $\{k_i\}_{i=1}^\infty$, such that
\[
\lim_{i\rightarrow\infty}\frac{|n_{k_i}|}{e^{-w(k_i)}}<\infty,
\]
 then there exists $\eps_0(n)$ such that for all $0<\eps\leq\eps_0$, the deformed boundary couldn't persist the coexistence of $1/2,1/3$-caustics.
\end{thm}
\begin{proof}
This conclusion is proved in Section \ref{3}, Theorem \ref{P-A}.
\end{proof}
\begin{cor}[Trigonometric Polynomial]
For any non degenerate deformation of the circle which can be formalized by $(*)$ and satisfying (\ref{rescale}), (\ref{rigid}) and (\ref{flux}), if $n(\theta)$ is an even trigonometric polynomial, then there exists $\eps_0(n)$ such that for all $0<\eps\leq\eps_0$, the deformed boundary couldn't persist the coexistence of $1/2,1/3$-caustics.
\end{cor}
\begin{proof}
This conclusion is proved in Section \ref{3}, Theorem \ref{A-A}.
\end{proof}
Let's define the {\bf action function} of $q-$periodic configuration by
\[
P_q(\theta, \Omega_\eps) =
\text{maximal perimeter of a } q-\text{gon inscribed into}
\]\[ \text{ the domain }\
\Omega_\eps\ \text{ starting and ending at}\ \theta,
\]
then 
\[
P_q(\theta,\Om_\eps)=-F_{1/q}(s(\theta))
\]
 if we consider the arc length variable $s$ as a function of the angle variable $\theta$. Recall that $s\in[0,2\pi]\rightarrow \theta\in[0,2\pi]$ is a diffeomorphism via the following:
\be\label{dif-arc}
s=\int_0^\theta \sqrt{\dot{r}^2_\eps(\tau)+r_\eps^2(\tau)}d\tau,
\ee
so $s(\theta+2\pi)=s(\theta)+2\pi$ can be easily achieved.
\begin{lem} \label{general-action}
Let $n \in T_q$, then the domain $\Omega_\eps$ (*) has
\be
P_q(\theta,\Omega_\eps) = P_q(\theta, \Omega_0) + \eps c+
\eps^2\Big{ [}2q\sin\frac{\pi}{q}m^{(q)}(\theta)+D_q(\theta,\eps)\Big{]},
\ee
for some $c$, where $$
m^{(q)}(\theta)=\sum_{k\in \Z} m_{kq} \exp ikq\theta
$$
is the averaging of $1/q-$frequency of $m(\theta)$.
\end{lem}
\begin{proof}
See section \ref{2} for details.
\end{proof}
Consider the Fourier expansion of $D_q$
\[
D_q(\theta,\eps)=\sum_{k\in \Z} D_{q,k}(\eps) \exp i\,k\theta,
\]
then the following property holds:
\begin{lem}\label{q-period}
 In the above notations, if $\gamma_\eps\in BZ_q$,
\[
D_q(\theta,\eps)=\sum_{k=mq,\ m\in \Z} D_{q,mq}(0) \exp i\,(mq\theta)+O(\eps),
\]
i.e. $D_q(\theta,\eps)$ is $1/q-$periodic of $\theta$.
\end{lem}

\begin{proof}
Since $\gamma_\eps\in BZ_q$, $P_q(\theta,\Om_\eps)$ is a constant. Due to previous Lemma, 
\[
2q\sin\frac{\pi}{q}m^{(q)}(\theta)+D_q(\theta,\eps)=const.
\]
Recall that $m^{(q)}(\theta)=\sum_{l}m_{ql}\exp({\bf i}ql\cdot \theta)$ contains only harmonics divisible by $q$, that means $D_q(\theta,\eps)$ has to contain only harmonics divisible by $q$ as well, i.e.
\[
D_q(\theta):=D_q(\theta,0)=\sum_{l\in\Z}D_{ql}\exp({\bf i}ql\cdot \theta).
\]
So $D_q(\theta)$ is $1/q-$periodic of $\theta$.
\end{proof}

%

\subsection{Obstruction to coexistence of two rational caustics}
\label{sec:obstruct}
In particular, for $q=2,3$ we want to
show that for all
$$
n\in (T_2 \cap T_3) \setminus 0
$$
the functions $\frac{3\sqrt3}{4} D_2(\theta,\eps)$ and $D_3(\theta,\eps)$ have
at least one Fourier harmonic divisible by $6$ whose coefficients are different. Once we did this, the contradiction with the action functions will lead to our main conclusion. 

\begin{cor} Once $\frac{3\sqrt3}{4}D_2(\theta)$
and $D_3(\theta)$ have a distinct Fourier harmonic whose index
is divisible by $6$, then in a $O(\eps^2)$ neighborhood of
$n(\theta)$ there is no domain having $1/2$ and $1/3$ caustic.
\end{cor}
\begin{proof}
Indeed,
\[
P_q (\theta,\Omega_\eps)=P_q(\theta,\Om_0)+c\eps + \eps^2 \left(D_q(\theta) 
+2q\sin \frac{\pi}{q}\,m^{(q)}(\theta)\right)+O(\eps^3)
\]
\[
P_p (\theta,\Omega_\eps)=P_p(\theta,\Om_0)+c\eps + \eps^2 \left(D_p(\theta)
+2p\sin \frac{\pi}{p}\,m^{(p)}(\theta)\right) +O(\eps^3).
\]
If we take $q=2$ and $p=3$, then the necessary condition to preserve $1/2$, $1/3$ caustics is that $P_2$, $P_3$ should be both constant, which leads to 
\[
D_2(\theta) 
+4\,m^{(2)}(\theta)=const,
\]
\[
D_3(\theta)
+3\sqrt 3\,m^{(3)}(\theta)=const.
\]
Once again we can take the averaging and get 
\[
\frac{3\sqrt3}{4}D_2^{(3)}(\theta)=D_3^{(2)}(\theta).
\]
This is a necessary condition for boundary $\Omega_\eps$ to preserve both $1/2$ and $1/3$ caustics.
\end{proof}
The Fourier coefficients of $D_q(\theta)$ are obtained through
a convolution of Fourier coefficients of $n$. We will compute
the corresponding formula in the next section.

\section{Evaluation of $P_q$ action function}\label{2}

Let's start with the circle $\Omega_0$ and the $q$-gon for some $q\ge 2$. Suppose the deformation in the polar coordinate has the form 
\[
\gamma_\eps =\{ r = 1+ \eps n(\theta)+\eps^2 m(\theta),\ \theta\in \T\}
\]
for small $\eps$ and 
$$
n(\theta)=\sum_{k\in \Z} n_k \exp 2\pi i k\cdot \theta\quad \&
\quad m(\theta)=\sum_{k\in \Z} m_k \exp 2\pi i k\cdot \theta.
$$
Consider a $q$-perimeter function $P_\eps (\theta,\Omega_\eps)$
as defined above. For $\eps=0$ we have
$$P_q(\theta,\Om_0)\equiv 2q\sin \frac{\pi}{q}.$$

Let $(z^0_0,\dots,z^0_{q-1})$ be the right $q$-gon, i.e. $\theta^0_k=\theta^0_0+2k\pi/q$.
For small $\eps$ we compute
\[
z_k^\eps = z_k^0+\eps \eta_k + \eps^2 \xi_k+O(\eps^3),\qquad
k=0,\dots,q-1.
\]
We postpone the computation of $\eta_k$ and $\xi_k$. Consider the $k$-th edge between $z_k$ and $z_{k+1}$. Taking a dot 
product of $z_k^\eps - z_{k+1}^\eps$ with itself we have  
{{
\[
|z_k^\eps - z_{k+1}^\eps |^2=
|z^0_k - z^0_{k+1}|^2+
\eps^2 |\eta_k-\eta_{k+1} |^2 \]\[
+2 \eps (z^0_k - z^0_{k+1})\cdot (\eta_k-\eta_{k+1})
+2 \eps^2 (z^0_k - z^0_{k+1})\cdot (\xi_k - \xi_{k+1})+O(\eps^3).
\]}}
Rewrite
\begin{align*}
  |z_k^\eps - z_{k+1}^\eps |
  &= |z^0_k - z^0_{k+1}| + \eps \frac{\langle z^0_k - z^0_{k+1},\eta_k - \eta_{k+1}\rangle}{|z^0_k-z^0_{k+1}|}\\
  &\phantom = - \frac {\eps^2}2 \frac{\langle z^0_k - z^0_{k+1},\eta_k - \eta_{k+1}\rangle^2}{
    |z^0_k - z^0_{k+1}|^3}
    +\frac {\eps^2}2 \frac{ |\eta_k-\eta_{k+1} |^2}{|z^0_k - z^0_{k+1}| } \\
  &\phantom = + \eps^2 \frac{\langle z^0_k - z^0_{k+1},\xi_k - \xi_{k+1}\rangle}{|z^0_k-z^0_{k+1}|} +O(\eps^3).
\end{align*}
Summing over $k$ we get
  \begin{align*}
    &P_q(\theta,\Om_\eps)  =P_q(\theta,\Om_0)
+  \eps \sum_{k=0}^{q-1} \Big[
 \frac{z^0_k - z^0_{k+1}}{|z^0_k-z^0_{k+1}|} \cdot (\eta_k - \eta_{k+1})
+O(\eps^2)\\
\phantom =
  + \eps\,&\frac{z^0_k - z^0_{k+1}}{|z^0_k-z^0_{k+1}|} 
\cdot  (\xi_k - \xi_{k+1})- \frac {\eps}2
 \frac{\langle z^0_k - z^0_{k+1},\eta_k - \eta_{k+1}\rangle^2}{
    |z^0_k - z^0_{k+1}|^3}
+\frac {\eps}2 \frac{ |\eta_k-\eta_{k+1} |^2}{|z^0_k - z^0_{k+1}| }
 \Big],
  \end{align*}
where $z_0^\eps/|z_0^\eps|=e^{i\theta}=z_0^0/|z_0^0|$ (for convenience we can interpret the unit vector in $\R^2$ as a complex number in $\mathbb C$, this greatly simplifies our formulas without confusion). Also we can denote the unit vector $\frac{z^0_k - z^0_{k+1}}{|z^0_k-z^0_{k+1}|}$ by $e_k$ for short. 
  \begin{align*}
    P_q(\theta,\Om_\eps)  &=P_q(\theta,\Om_0)
+  \eps \sum_{k=0}^{q-1} \Big[
 e_k\, \cdot (\eta_k - \eta_{k+1} )
+O(\eps^2)\\
\phantom =
  + \eps\,&e_k
\cdot  (\xi_k - \xi_{k+1})- \frac {\eps}2
\frac{\langle e_k ,\eta_k - \eta_{k+1}\rangle^2}{
|z^0_k - z^0_{k+1}|}
+\frac {\eps}2 \frac{ |\eta_k-\eta_{k+1} |^2}{|z^0_k - z^0_{k+1}| }
 \Big].
  \end{align*}
{Recall that $z_q^\eps=z_0^\eps$ and $z_q^0=z_0^0$, which leads to $\eta_0=\eta_q$ and $\xi_0=\xi_q$.} Combining $e_k\cdot \eta_k$ and $-e_{k-1}\cdot \eta_k$ 
and observing that $(e_k-e_{k-1})$ is the outer normal vector
to the boundary $\partial \Om_0$ at $z^0_k$, then we have 
\begin{align*}
    P_q(\theta,\Om_\eps)  &=P_q(\theta,\Om_0)
+  \eps \sum_{k=0}^{q-1} \Big[ (e_k-e_{k-1})\, \cdot 
(\eta_k +\eps \xi_k) 
+O(\eps^2)\\
\phantom =
  & - \frac {\eps}2
\frac{[e_k \cdot (\eta_k - \eta_{k+1})]^2}{
|z^0_k - z^0_{k+1}|^1}
+\frac {\eps}2 \frac{ |\eta_k-\eta_{k+1} |^2}{|z^0_k - z^0_{k+1}| }
 \Big]= \\
  &=P_q(\theta,\Om_0)
+  \eps \sum_{k=0}^{q-1} \Big[ (e_k-e_{k-1})\, \cdot 
(\eta_k^\perp +{\eps \xi_k^\perp}) 
+O(\eps^2)\\
\phantom =
  & - \frac {\eps}2
\frac{[e_k \cdot (\eta_k - \eta_{k+1})]^2}{
|z^0_k - z^0_{k+1}|}
+\frac {\eps}2 \frac{ |\eta_k-\eta_{k+1} |^2}{|z^0_k - z^0_{k+1}| }
 \Big],
  \end{align*}
where $\eta_k^\perp$ (resp. $\xi_k^\perp$)
is the component of $\eta_k$ (resp. $\xi_k$) perpendicular to
$(z^0_{k-1} - z^0_{k+1})$ or, equivalently, the component of $\eta_k$
(resp. $\xi_k$) normal to the boundary at $z^0_k$.

\subsection{The leading term in the case of the circle}
Consider the leading term of the expansion 
\[
P_q(\theta,\Om_\eps)=P_q(\theta,\Om_0)
+  \eps \sum_{k=0}^{q-1} \Big[ (e_k-e_{k-1})\, \cdot 
\eta_k^\perp \Big] +O(\eps^2). 
\]
Notice that 
\[
e_k-e_{k-1}={2 \sin \frac {\pi} {q}\ {\bf n}(\theta_k^0)},\quad k=1,\cdots,q
\]
 for $q\geq2$.
 Here the bold $\mathbf n(\cdot)$ is the normal vector, not the $1^{st}$ jet of the deformation ! Then we get  
\be\label{1st-order}
\eta_k^\perp =\langle\eta_k,{\bf n}(\theta_k^0)\rangle{\bf n}(\theta_k^0)=n(\theta^0_k){\bf n}(\theta^0_k)+O(\eps).
\ee
{This is because 
\begin{eqnarray}\label{identity}
z_k^\eps&=&z_k^0+\eps\eta_k+\eps^2\xi_k+O(\eps^3)\nonumber\\
&=&r_\eps(\theta_k^\eps)e^{i\theta_k^\eps}\nonumber\\
&=&[1+\eps n(\theta_k^\eps)+\eps^2m(\theta_k^\eps)]e^{i\theta_k^\eps},
\end{eqnarray}
and
\[
\theta_k^\eps=\theta_k^0+\eps\theta_k^1+O(\eps^2),
\]
which leads to the $1^{st}$ order equation:
} 
\begin{lem}\label{1st-order}
The $1^{st}$ order estimate of $P_q(\theta,\Om_\eps)$ obeys
\[
P_q(\theta,\Om_\eps)=P_q(\theta,\Om_0)
+ 2 \eps q\sin \frac {\pi}{ q} n^{(q)}(\theta) 
+O(\eps^2),\quad q\geq2.
\]
\end{lem}

\subsection{The second order in the case of the circle}

In order to get a $2^{nd}$ order estimate for $P_q(\theta,\Om_\eps)$, we need to get the exact expression of $\eta_k$ and $\xi_k$. Due to (\ref{identity}), we get 
\[
\eta_k(\theta_k^0)=\eta_k^{\perp}+\eta_k^{\parallel}=n(\theta_k^0)e^{i\theta_k^0}+\theta_k^1e^{i(\pi/2+\theta_k^0)}+O(\eps)
\]
and
\[
\xi_k(\theta_k^0)=\Big{[}-\frac12{\theta_k^1}^2+n'(\theta_k^0)\theta_k^1+m(\theta_k^0)\Big{]}e^{i\theta_k^0}+n(\theta_k^0)\theta_k^1e^{i(\pi/2+\theta_k^0)}+O(\eps).
\]
Recall that $\theta_0^1=\theta_q^1=0$, then
\be\label{perp}
\Big\langle {\bf n}^\eps(\theta_k^\eps),\frac{z_{k+1}^\eps-z_k^\eps}{|z_{k+1}^\eps-z_k^\eps|}-\frac{z_{k-1}^\eps-z_k^\eps}{|z_{k-1}^\eps-z_k^\eps|}\Big\rangle=0,\quad k=1,\cdots,q-1.
\ee
On the other side, 
\be
{\bf n}^\eps(\theta_k^\eps)&=&e^{i\pi/2}\cdot\dot{\gamma}_\eps(\theta)\Big{|}_{\theta=\theta_k^\eps}\nonumber\\
&=&e^{i(\pi/2+\theta_k^\eps)}\Big{[}\eps n'(\theta_k^\eps)+\eps^2m'(\theta_k^\eps)\Big{]}-e^{i\theta_k^\eps}\Big{[}1+\eps n(\theta_k^\eps)+\eps^2m(\theta_k^\eps)\Big{]}\nonumber\\
&=&e^{i(\pi/2+\theta_k^0)}\Big{[}\eps n'(\theta_k^0)+\eps^2 n''(\theta_k^0)\theta_k^1+\eps^2m'(\theta_k^0)\Big{]}-\Big{(}\eps\theta_k^1 e^{i\theta_k^0}+\nonumber\\
& &\frac 12 \eps^2{\theta_k^1}^2e^{i(\pi/2+\theta_k^0)}\Big{)}\cdot\Big{[}\eps n'(\theta_k^0)+\eps^2 n''(\theta_k^0)\theta_k^1+\eps^2m'(\theta_k^0)\Big{]}\nonumber\\
& &-\Big{[}e^{i\theta_k^0}+\eps\eta_k(\theta_k^0)+\eps^2\xi_k(\theta_k^0)\Big{]}+O(\eps^3)\nonumber\\
&=&-{\bf n}^0(\theta_k^0)+\eps n'(\theta_k^0){\bf t}^0(\theta_k^0)-\eps\eta_k(\theta_k^0)+O(\eps^2)
\ee
and
\be
z_{k+1}^\eps-z_{k-1}^\eps &=&\Big{[}z_{k+1}^0-z_{k-1}^0\Big{]}+\eps\Big{[}\eta_{k+1}(\theta_{k+1}^0)-\eta_{k-1}(\theta_{k-1}^0)\Big{]}+O(\eps^2),\nonumber
\ee
where we used the estimate 
\[
e^{i(\phi+\eps\psi)}-e^{i\phi}=\eps\psi e^{i(\pi/2+\phi)}-\frac 12 \eps^2\psi^2 e^{i\phi}+O(\eps^3),\quad\forall \phi,\psi\in\T.
\]
Then we turn back to (\ref{perp}) and get 

\be
\langle{\bf n}^0(\theta_k^0),\eta_{k+1}(\theta_{k+1}^0)-\eta_{k-1}(\theta_{k-1}^0)\rangle+\langle\eta_k(\theta_k^0)-n'(\theta_k^0){\bf t}^0(\theta_k^0),z_{k+1}^0-z_{k-1}^0\rangle\nonumber\\
=\Big\langle{\bf n}^0(\theta_k^0),\langle e_k,\eta_{k+1}-\eta_k\rangle e_k-\langle e_{k-1},\eta_{k-1}-\eta_k\rangle e_{k-1}\Big\rangle.
\ee

Simplify it:
\be
\Big(n(\theta_{k+1}^0)-n(\theta_{k-1}^0)\Big)\cos\frac{2\pi}{q}+\Big(2\theta_k^1-\theta_{k+1}^1-\theta_{k-1}^1-2n'(\theta_k^0)\Big)\sin\frac{2\pi}{q}=\nonumber\\
\frac12\Big[\big(\cos\frac{2\pi}{q}-1\big)\big(n(\theta_{k+1}^0)-n(\theta_{k-1}^0)\big)-\sin\frac{2\pi}{q}(\theta_{k+1}^1-\theta_{k-1}^1)\Big]
\ee
we  finally get a triple-diagonal linear equation group:
\be\label{triple}
4\theta_k^1-\theta_{k+1}^1-3\theta_{k-1}^1=4n'(\theta_k^0)-\frac{n(\theta_{k+1}^0)-n(\theta_{k-1}^0)}{\tan\frac {\pi} q},\quad k=1,\cdots,q-1,
\ee
with $\theta_0^1=\theta_q^1=0$. This is a general formula holding for all $q\geq3$.

\begin{rmk}\label{theta-1-est}
For $q=2$, $\theta_1^1=n'(\theta+\pi)-n'(\theta)$, $\theta_0^1=0$. Mention that in this case (\ref{triple}) is invalid, but we can use that ${\bf n}^\eps(\theta_0^\eps)\parallel{\bf n}^\eps(\theta_1^\eps)$.\\
For $q=3$, we can solve (\ref{triple}) by
\[
\theta_1^1=\frac{4}{13}[n'(\theta+4\pi/3)+4n'(\theta+2\pi/3)]+\frac{1}{13\sqrt3}[3n(\theta)+n(\theta+2\pi/3)-4n(\theta+4\pi/3)],
\]
\[
\theta_2^1=\frac{4}{13}[3n'(\theta+2\pi/3)+4n'(\theta+4\pi/3)]+\frac{1}{13\sqrt3}[-n(\theta)-3n(\theta+4\pi/3)+4n(\theta+2\pi/3)].
\]
\end{rmk}

Recall that 
\begin{align*}
    P_q(\theta,\Om_\eps)    &=P_q(\theta,\Om_0)
+  \eps \sum_{k=0}^{q-1} \Big[ (e_k-e_{k-1})\, \cdot 
(\eta_k^\perp +{\eps \xi_k^\perp}) 
+O(\eps^2)\\
\phantom =
  & - \frac {\eps}2
\frac{[e_k \cdot (\eta_k - \eta_{k+1})]^2}{
|z^0_k - z^0_{k+1}|}
+\frac {\eps}2 \frac{ |\eta_k-\eta_{k+1} |^2}{|z^0_k - z^0_{k+1}| }
 \Big],
  \end{align*}

 Each $\eta_k=\eta_k^\perp+\eta_k^\parallel$. The above calculations 
 show that 
 \be
 \eta_k=n(\theta_k^0) \cdot {\bf n}(\theta^0_k)+{\theta_k^1\cdot{\bf t}(\theta_k^0)}+O(\eps)
 \ee
 with
\be
 \theta_k^1=\sum_{j=0}^{q-1}c^k_jn'(\theta_j^0)+d^k_jn(\theta_j^0),\quad k=1,\cdots,q-1
 \ee
which is solved from (\ref{triple}).
 Notice also that 
 \[
 \xi_k(\theta_k^0)=\Big{[}-\frac12{\theta_k^1}^2+n'(\theta_k^0)\theta_k^1+m(\theta_k^0)\Big{]}\cdot{\bf n}(\theta^0_k)+n(\theta_k^0)\theta_k^1\cdot{\bf t}(\theta^0_k)+O(\eps).
 \]
 So
 \[
\xi_k^\parallel =\sum_{j=0}^{q-1}c^k_jn(\theta_k^0) n'(\theta_j^0)+d^k_jn(\theta_k^0) n(\theta_j^0)
 \]
 and
 \be
 \xi_k^\perp&=&-\frac12 \sum_{i,j=0}^{q-1}(c^k_jn'(\theta_j^0)+d^k_jn(\theta_j^0))\cdot(c^k_in'(\theta_i^0)+d^k_in(\theta_i^0))\nonumber\\
 & &+\sum_{j=0}^{q-1}c^k_jn'(\theta_k^0) n'(\theta_j^0)+d^k_jn'(\theta_k^0) n(\theta_j^0)+m(\theta_k^0).
 \ee
 Substitute them in $P_q$ and we get the $\eps^2$-term by
\be\label{2nd}
\sum_{k=0}^{q-1} \Bigg\{{2\sin\frac{\pi}{q}}\xi_k^\perp
  - \frac{\Big{[}(1-\cos\frac {2\pi} q)(n(\theta_k^0)+n(\theta_{k+1}^0))+\sin\frac{2\pi}q(\theta_{k+1}^1-\theta_k^1)\Big{]}^2}{
{16\sin^3 \frac \pi{q}}}\nonumber\\
+\frac{ n^2(\theta_k^0)+{\theta_k^1}^2+n^2(\theta_{k+1}^0) +{\theta_{k+1}^1}^2}{4\sin \frac \pi{q} }\nonumber\\
+\frac{-2\Big [n(\theta_k^0)n(\theta_{k+1}^0)+\theta_k^1\theta_{k+1}^1\Big ]\cos\frac{2\pi}q+2\sin\frac{2\pi}{q}\Big[n(\theta_k^0)\theta_{k+1}^1-n(\theta_{k+1}^0)\theta_k^1\Big]}{4\sin \frac \pi{q} }
 \Bigg\}.
  \ee
Substituting for some computable matrix $A_{q\times q},B_{q\times q},C_{q\times q}$ we have 
\be \label{factors}
D_q(\theta,\eps)=\sum_{i,j=0}^{q-1}&\ \langle n'(\theta_i),B_{(q)} \cdot  n(\theta_j)\rangle+\langle n'(\theta_i),C_{(q)} \cdot  n'(\theta_j)\rangle \nonumber\\ 
 +&\langle n(\theta_i),A_{(q)} \cdot  n(\theta_j)\rangle+
 O(\eps).
  \ee
  
Recall that $m^{(q)}(s)=\sum_{k\in \Z} m_{kq}\exp (iqks)$ and $n(s)=\sum_{k\in \Z}n_k \exp iks$, we have 

\begin{align*} 
 n(\theta_i)n(\theta_j)&=&\sum_{k,l\in \Z} n_{k}n_l \exp ({\bf i}[ki+lj]\frac{2\pi}{q})\exp ({\bf i}[k+l]\theta). \qquad \qquad \qquad \qquad \quad\quad\quad\;
 \\
 n'(\theta_i)n(\theta_j)&=&\sum_{k,l\in \Z} n_{k}n_l k\exp ({\bf i}[ki+lj]\frac{2\pi}{q})\exp ({\bf i}[k+l]\theta). \qquad \qquad \qquad \qquad \quad\quad\quad
 \\
 n'(\theta_i)n'(\theta_j)&=&\sum_{k,l\in \Z} n_{k}n_l kl\exp ({\bf i}[ki+lj]\frac{2\pi}{q})\exp ({\bf i}[k+l]\theta). \qquad \qquad \qquad \qquad \quad\quad\quad
  \end{align*}
In the Appendix you can find a detailed calculation of $D^{(2)}$ and $D^{(3)}$.
\begin{thm}
Suppose $\gamma_\eps$ can preserve $1/2$ and $1/3$ caustics, then for $q=2,3$, the second order estimate can be achieved by
\be\label{d2}
D_2(\theta)&=&{\frac{3n'^2(\theta+\pi)-n'^2(\theta)}{2}-n'(\theta)n'(\theta+\pi)}\nonumber\\
&=&\textr{2n'^2(\theta)}\nonumber\\
&=& 2 \sum_{k,l\in\mathbb{Z}}kln_k n_l \exp({\bf i}[k+l]\theta)
\ee
because \textr{$n^{(2)}=const$} leads to \textr{$n'(\theta+\pi)+n'(\theta)=0$} for all  $\theta\in\T$ and
\be\label{d3-un}
D_3(\theta)&=&\sum_{k,l\in\mathbb Z}\frac{1}{13\sqrt 3}\Big[4-2e^{\frac{2\pi}{3}k}-2e^{\frac{4\pi}{3}k}+(3\sqrt3-2)e^{\frac{2\pi}{3}l}+(4+3\sqrt3 l+36kl)e^{\frac{4\pi}{3}(k+l)}\nonumber\\
& &+(6kl-2-6\sqrt3 l)e^{\frac{4\pi}{3}k+\frac{2\pi}{3}l}-(2+3\sqrt3 l)e^{\frac{4\pi}{3}l}+(-2+6\sqrt3 l+30kl)e^{\frac{2\pi}{3}k+\frac{4\pi}{3}l}\nonumber\\
& &+(4-3\sqrt3l+36kl)e^{\frac{4\pi}{3}(k+l)}\Big]n_kn_l\exp({\bf i}(k+l)\theta).
\ee
\end{thm}

 Denote by $q\Z:=\Big\{qn\Big|n\in\Z\Big\}$ for $q\in\Z_+$ and $\cZ_{2,3}=\Z\backslash(2\Z\cup3\Z)$, then due to a simple arithmetic deduction, we get 
 \begin{lem}
$ \cZ_{2,3}=\{6l\pm1|l\in\Z\}$.
 \end{lem}

\begin{cor}\label{coeffi}
For $\gamma_\eps\in BZ_2\cap BZ_3$, we get conditional $2^{nd}$ order estimate by
\be\label{d2}
D_2(\theta)=\sum_{\substack{k\geq l}}2n_kn_lc^{(2)}(k,l)\exp({\bf i}(k+l)\theta),
\ee
with
\[
\textr{c^{(2)}(6p+1,6q+1)=0,\quad\forall p,q\in\Z},
\]
\[
\textr{c^{(2)}(6p-1,6q-1)=0,\quad\forall p,q\in\Z},
\]
\[
\textr{c^{(2)}(6p+1,6q-1)=2(6p+1)(6q-1),\quad\forall p,q\in\Z},
\]
\[
\textr{c^{(2)}(6p-1,6q+1)=2(6p-1)(6q+1),\quad\forall p,q\in\Z},
\]
 and
\be\label{d3}
D_3(\theta)=\sum_{\substack{k\geq l}}2n_kn_lc^{(3)}(k,l)\exp({\bf i}(k+l)\theta),
\ee
with
\[
c^{(3)}(6p+1,6q+1)=0,\quad\forall p,q\in\Z,
\]
\[
c^{(3)}(6p-1,6q-1)=0,\quad\forall p,q\in\Z,
\]
\[
c^{(3)}(6p+1,6q-1)=\frac{1}{13\sqrt3}\Big[(36+108kl)+{\bf i}(27(l+k)+24\sqrt3kl)\Big],\quad\forall p,q\in\Z,
\]
\[
c^{(3)}(6p-1,6q+1)={\frac{1}{13\sqrt3}\Big[(36+108kl)-{\bf i}(27(l+k)+24\sqrt3kl)\Big]},\quad\forall p,q\in\Z.
\]
\end{cor}
\begin{proof}
Due to Lemma \ref{q-period}, we get $n^{(2)}=n^{(3)}=0$, which implies $n_k=0$ for all $k\in 2\Z\cup 3\Z$. Moreover, \textr{$D_2(\theta)=2n'^2(\theta)=-2n'(\theta)n'(\theta+\pi)$} is naturally $\pi-$periodic and $D_3(\theta)=D_3^{(3)}(\theta)$ due to a detailed computation in Appendix. This lead to the coefficient equalities.
\end{proof}
From previous corollary we also get the following property:
\begin{lem}\label{n-deg}
Let's simplify the notation by
\be\label{gene-D}
D_q(\theta)=\sum_{k\geq l}n_kn_lc^{(q)}(k,l)\exp({\bf i}(k+l)\theta),
\ee
then the coefficient $c^{(q)}(k,l)\in \mathbb C$ satisfies
\[
c^{(2)}(k,l)\nparallel c^{(3)}(k,l).
\]
\end{lem}
\begin{proof}
This is a direct arithmetic observation.
\end{proof}

\section{the harmonic analysis for $n\in T_2\cap T_3$}\label{3}
In this section we analyze the harmonic behaviour of (\ref{factors}). Since $n\in T_2\cap T_3$, then the Fourier coefficients satisfy
\[
n_{2l}=n_{3l}=0,\quad\forall l\in\mathbb{Z}
\]
due to Lemma \ref{1st-order}. Moreover, suppose the Fourier expansion of $D_q(\theta)$ satisfies
\[
D_q(\theta)=\sum_{k\in\Z}D_{q,k}e^{ik\theta},
\]
then from the second order estimate we get the following harmonic equalities: for all $l\in\Z$, 
\be\label{harmonic}\label{pyramid-dia}
D^{(2)}_{2l+1}&=&0,\quad\quad\quad\quad(\spadesuit)\nonumber\\
D^{(3)}_{3l+1}&=&0,\quad\quad\quad\quad(\clubsuit)\nonumber\\
D^{(3)}_{3l+2}&=&0,\quad\quad\quad\quad(\heartsuit)\nonumber\\
4D^{(2)}_{6l}&=&3\sqrt3 D^{(3)}_{6l}.\quad(\diamondsuit) 
\ee
As long as one of previous equalities is failed for non trivial $n(\theta)$, we would get $\gamma_\eps\notin BZ_2\cap BZ_3$ and prove the Projected Theorem.
\begin{lem}
$(\spadesuit), (\clubsuit)$ and $(\heartsuit)$ naturally hold for $n\in T_2\cap T_3$.
\end{lem}
\begin{proof}
This is a direct conclusion from Corollary \ref{coeffi}.
\end{proof}
\vspace{10pt}
\subsubsection{the polynomial case: pyramid type harmonic equations} 

Suppose $n(\theta)$ is a trigonometric polynomial with the degree be $N$, i.e.
\[
n(\theta)=\sum_{|k|\leq N}n_ke^{{\bf i}k\theta},
\]
then (\ref{pyramid-dia}) becomes a `pyramid' type equation group with quadratic monomials. Benefit from this structure we can prove the following:
\begin{thm}[Polynomial Avalanche]\label{P-A}
For even trigonometric polynomial $n(\theta)\neq0$, the deformative boundary couldn't survive the coexistence of $1/2,1/3$ caustics, as long as $0<\eps\leq\eps_0(n)$.
\end{thm}
\begin{proof}
We can prove this by contradiction. Suppose $1/2,1/3$ caustics coexist, then (\ref{pyramid-dia}) should hold. Without loss of generality, we can assume $6P+1$ is the largest integer in $\cZ_{2,3}$ which doesn't exceed $N$, due to Lemma \ref{n-deg}, 
\be\label{mono-deg}
n_{6p+1}\cdot n_{6q-1}=0,\quad\forall q,p\in\Z,\; -P\leq q,p\leq P.
\ee
Now if make an {\bf opposite pair} by $(-a,a)$ of any integer $a\in \cZ_{2,3}\cap[-N,N]$, then we can make the following claim:\\

{\bf Claim:} There exists only one opposite pair $(-a,a)$ for $a\in \cZ_{2,3}\cap[-N,N]$, such that $n_{-a}\cdot n_a\neq0$.\\

The proof of this claim is straightforward. Without loss of generality, we assume $a=6\alpha+1$ $\alpha\in\Z$. If there exists another pair $(-b,b)$ disapproves this claim, then $b=6\beta\pm 1$ with $\beta\in\Z$. For $b=6\beta+1$, we can get $n_{6\alpha+1}\cdot n_{-6\beta-1}\neq0$; For $b=6\beta-1$, we can get $n_{6\alpha+1}\cdot n_{6\beta-1}\neq0$; Anyway this disobeys (\ref{mono-deg}) and leads to a contradiction.\\

Recall that $n(\theta)$ is even, so $n_k=n_{-k}$ for all $k\in\Z$. Due to the claim, there will be only one couple $(-a,a)$, such that $n_{-a}\cdot n_a\neq0$. But from (\ref{rigid}) we know that $n_a+n_{-a}=0$ should hold simultaneously. This implies that $n_a=n_{-a}=0$ and $n(\theta)=0$.
\end{proof}

\subsubsection{the general analytic case: avalanche caused by a quantitative control of error terms}

In this section we generalize the idea in Theorem \ref{P-A}, and prove a similar result for general analytic $n(\theta)$. Denote by $C^w(\T,\R,\rho)$ the set of all analytic functions with radius $\rho$, then it's a Banach space under the analytic norm $\|\cdot\|_\rho$. The following estimate of the Fourier coefficients holds:
\begin{lem}\label{rho-decay}
If $f(x)\in C^w(\T,\R,\rho)$, then $f(x)=\sum_kf_ke^{{\bf i}kx}$ with
\[
|f_k|\leq\|f\|_\rho e^{-|k|\rho},\quad k\in\Z.
\]
\end{lem}

Notice that previous Lemma is not always the optimal estimate of the Fourier coefficients for all functions in $C^w(\T,\R,\rho)$, so we can use the following procedure to find the {\bf slowest decaying coefficient sequence}, and define the corresponding {\bf modular function}.\\

For $n(\theta)=\sum n_k e^{{\bf i}k\theta}$ consisting of infinitely many terms, since Lemma \ref{rho-decay}
is available, then we can pick $k_1\in\Z$ be the index satisfying
\[
|n_{k_1}|=\sup\Big\{|n_k|\Big| k\in\Z\Big\}.
\]
If there are several candidate index, we can choose the one with the greatest absolute value. Based on the same principle, we can choose $k_2\in\Z$, such that 
\[
|n_{k_2}|=\sup\Big\{|n_k|\Big|k\in\Z,|k|>|k_1|\Big\}.
\]
Repeat this process we can get a sequence $\{k_i\}_{i=1}^\infty$, and the corresponding coefficient sequence $\{n_{k_i}\}_{i=1}^\infty$ is just the slowest decaying coefficient sequence of $n(\theta)$.
\begin{defn}
We call a smooth decreasing, positive function $w:[0,+\infty)\rightarrow \R^+$ {\bf the modular function} of $n(\theta)$, if $w(|k_i|)=|n_{k_i}|$ for all $i=1,2,\cdots$.
\end{defn}
\begin{rmk}
Notice that the modular function $w(x)$ is not uniquely defined, but any two modular functions $w_1,w_2$ corresponding to the same $n(\theta)$ should satisfy:
\[
\lim_{i\rightarrow\infty}\frac{w_1(i)}{w_2(i)}=1.
\]
Moreover, for $n(\theta)\in C^w(\T,\R,\rho)$, 
\[
\lim_{i\rightarrow\infty}\frac{w(i)}{i}\geq\rho.
\]
\end{rmk}

\vspace{10pt}

For any $ L\in\{k'_i\}_{i=1}^\infty$, which is a positive subsequence of $\{k_i\}_{i=1}^\infty$, there exists a maximal $P\in\Z_+$, such that 
\[
L=6P+1\text{\; or\;} 6P-1.
\]
Without loss of generality, we just need to consider the first case. The following approximated equations can be derived from ($\diamondsuit$):
\be\label{ap-eq}
\Bigg| \sum_{\substack{k+l=6K\\|l|,|k|\leq L\\k>l\\k,l\in\cZ_{2,3}}}\big[c^{(2)}(k,l)-c^{(3)}(k,l)\big]n_kn_l\Bigg|\leq\cE,\quad -2P\leq K\leq2P.
\ee
Here we use $\cE\sim O(e^{- w(L)}L^2)$ is a Fourier reminder term. We can use the notation $c^{(2)-(3)}(k,l)=c^{(2)}(k,l)-c^{(3)}(k,l)$ for short.\\

Denote by
\[
\cN_{\diamondsuit}^L(K):=\{(k,l)\in \cZ_{2,3}\times\cZ_{2,3}|k+l=6K, k>l,|k|,|l|\leq L\}
\]
for $-2P\leq K\leq 2P$. This Lemma reveals the 'pyramid' structure of the main part of (\ref{ap-eq}):
\begin{lem}\label{pyramid}
For a fixed $K\in\Z$ with $|K|\leq2P$, $\sharp\cN^L_{\diamondsuit}(K)=(1+2P-|K|)$.
\end{lem}
\begin{proof}
To make $k+l=6K$, $k> l$ is always true. Moreover, we can assume $k=6\alpha\pm1$, $l=6\beta\mp1$, with $\alpha,\beta\in\Z$. 
So we ascribe the problem to 
\[
\alpha+\beta=K,\;-P\leq\beta\leq\alpha\leq P.
\]
So for $K>0$, the number of all possible $\alpha$ is $2P-K+1$. For $K<0$, the number of all possible $\beta$ is $2P+K+1$. Then we deduce a unified estimate by $2P-|K|+1$, which is the number for all the possible couple $(\alpha,\beta)$. 
\end{proof}
\begin{rmk}
From previous analysis, we can extra get $\sharp\cN_{\diamondsuit}^L(K)=\sharp\cN_{\diamondsuit}^L(-K)$.
\end{rmk}
Now let's explore the mechanism how the variables $n_k,n_l$ relate with each other for $\cN_\diamondsuit^L(K)$:
\begin{defn}
We define the {\bf generation} of $\cZ_{2,3}\cap[-L,L]$ by
\[
\cG(k):=P+1-\max\Big\{\lceil\frac{|k|}{6}\rceil,\lfloor\frac{|k|}{6}\rfloor\Big\}
\]
\end{defn}

\begin{lem}\label{couple}
For $K-1>0$, $\pi_1\cN_{\diamondsuit}^L(K-1)\subset\pi_1\cN_{\diamondsuit}^L(K)\cup\pi_2\cN_{\diamondsuit}^L(K)$; For $K<0$, $\pi_1\cN_{\diamondsuit}^L(K)\subset\pi_1\cN_{\diamondsuit}^L(K-1)\cup\pi_2\cN_{\diamondsuit}^L(K-1)$. Here $\pi_i$ is the coordinate projection to the corresponding component, $i=1,2$.
\end{lem}
\begin{lem}\label{d-index}
$\forall (k,l)\in\cN_{\diamondsuit}^L(K)$ with $0<K\leq 2P$, 
\[
\cG(k)+\cG(l)=2P-K+2,\quad\text{if\;} k>0\;\text{and\;}l>0
\]
and
\[
\cG(k)-\cG(l)=K-2P,\quad\text{if\;} l<0.
\]
\end{lem}
We can easily prove Lemma \ref{pyramid} and Lemma \ref{couple} from observation. To prove the following Lemma, let's assume $\lim_{i\rightarrow\infty}\frac{w(i)}{2^i}=\infty$ and $n(\theta)$ be even first:\begin{lem}\label{super-ex}
$\forall (k,l)\in\cN_{\diamondsuit}^L(K)$ with $-2P<K\leq 2P$ and $K\neq0$,
\be\label{sup-ex}
|c^{(2)-(3)}(k,l)|\cdot|n_k\cdot n_l|\lesssim \sqrt[2^{2P-|K|}]{e^{-w(L)}}\cdot L^{\sum_{i=0}^{2P-|K|}\frac{3}{2^i}}.
\ee
\end{lem}
\begin{proof}
Because $n(\theta)$ is even, so the modular function $w(i)$ is even for $i\in\Z$ and we just need to prove this Lemma for $0<K\leq 2P$. Let's start from the top level, i.e. $K=2P$,
if we denote by 
\[
\Delta_K=\max\Big\{|c^{(2)-(3)}(k,l)|\cdot|n_kn_l|\Big|(k,l)\in\cN_{\diamondsuit}^L(K)\Big\},\quad 0<K\leq2P,
\]
then
\[
\Delta_{K}\leq L^3\sqrt{\Delta_{K+1}}
\]
due to the pyramid structure of (\ref{ap-eq}).
Iterate this inequality we get (\ref{sup-ex}) for all $0<K\leq 2P$. Due to the symmetry of $n(\theta)$ we can generalize to $-2P\leq K<0$.
\end{proof}

\begin{thm}[Analytic Avalanche]\label{A-A}
For even analytic $n(\theta)\neq0$ with the modular function $w(x)$ satisfying
\[
\lim_{i\rightarrow\infty}\frac{w(i)}{2^i}=\infty,
\]
there exists $\eps_0(n)$ depending on $n(\theta)$, such that for all $0<\eps\leq\eps_0(n)$, the deformative boundary couldn't survive the coexistence of $1/2,1/3$ caustics.
\end{thm}
\begin{proof}
Similar with Theorem \ref{P-A}, we make the following approximated claim:\\

{\bf Claim:} There exists only one opposite pair $(-a,a)$ for $a\in \cZ_{2,3}\cap[-N,N]$, such that $\inf\{|n_{-a}|, |n_a|\}\geq e^{-\frac{w(L)}{2^{L/3}}}L^7$.\\

We can prove this claim by contradiction. Without loss of generality, we assume $a=6\alpha+1$ $\alpha\in\Z$. If there exists another pair $(-b,b)$ disapproves this claim, then $b=6\beta\pm 1$ with $\beta\in\Z$. For $b=6\beta+1$, we can get $|n_{6\alpha+1}\cdot n_{-6\beta-1}|\geq e^{-\frac{w(L)}{2^{L/3}}}L^7$; For $b=6\beta-1$, we can get $|n_{6\alpha+1}\cdot n_{6\beta-1}|\geq e^{-\frac{w(L)}{2^{L/3}}}L^7$. Anyway this contradicts with Lemma \ref{super-ex} by taking $L\rightarrow\infty$.\\

Recall that $n(\theta)$ is even and obeys (\ref{rigid}), 
\be
2|n_a|=|n_{-a}+n_a|=|\sum_{i\neq\pm a}n_i|&\leq&\sum_{i\neq\pm a}|n_i|\nonumber\\
&\leq&2\sum_{\substack{i>L}}|n_i|+2\cdot\frac{L^8}{6} e^{-\frac{w(L)}{2^{L/3}}}\nonumber\\
&\leq&2\sqrt\cE+\frac{L^8}{3}e^{-\frac{w(L)}{2^{L/3}}},
\ee
then 
\[
|n(\theta)|_{C^0}\leq\sum_{i\in\Z}|n_i|\leq4Le^{-w(L)}+\frac{2L^8}{3} e^{-\frac{w(L)}{2^{L/3}}}\leq4Le^{-w(L)}+\frac{2L^8}{3}\sqrt[2^{P+1}]L\cdot e^{-2^{\frac{2L}{3}}}
\]
and let $L\rightarrow0$, $|n(\theta)|_{C^0}\rightarrow 0$, which contradicts with the assumption $|n|_{C^0}=1$.
\end{proof}
\vspace{10pt}

\section{Further comments and heuristic improvement }\label{4}

Here we list some facts that guides our further exploration towards the Projected Conjecture. Recall that 
(\ref{triple}) and (\ref{2nd}) supply as a universal formula for all $q\geq2$, so does Lemma \ref{general-action} and Lemma \ref{q-period}. That gives us chance to propose a similar Conjecture:\\

\noindent{\bf the Elliptic Projected Conjecture:} In a $C^r$ ($r=2,\cdots,\infty,w$) neighborhood of the circle there is no other billiard domain preserving both $1/3$ and $1/5$ caustics.\\

The strategy to prove this elliptic conjecture is more or less the same with the case of $1/2$ and $1/3$ caustics. But instead, some arithmetic properties related with the harmonics of $n(\theta)$ will be changed, including the exact form of (\ref{ap-eq}).\\

Another aspect we could do is to generalize our main conclusion to general analytic function space $C^w(\T,\R,\rho)$, or even finite smooth space $C^r(\T,\R)$. The crucial lies in (\ref{ap-eq}), which is a homogeneous quadratic equation group of pyramid type. If we can make a better error control as solving it, we can reduce the decaying speed of the modular function $w(i)$. Some evidence indeed impies so:
\begin{lem} If we impose that the coefficients of $n(\theta)$ in the same generation are equivalent, i.e.
\[
\frac1c|n_{6P-1}|\leq |n_{6P+1}|\leq c|n_{6P-1}|,\quad\forall P\in\Z \text{\;holds\;} \text{for some $c\sim O(1)$},
\]
then we just need 
\[
\lim_{L\rightarrow\infty}\frac{w(L)}{L^2\ln L}=\infty.
\]
\end{lem}
Here we just give the essence for the proof: the condition we impose is actually for reducing the dimension of (\ref{ap-eq}). The similar idea holds, if we impose that more caustics preserve:\\

{\it Que: In a $C^r$ ($r=2,\cdots,\infty,w$) neighborhood of the circle there is no other billiard domain
of constant width, and preserving $1/3$, $1/5$ caustics.}\\

{\it Que: For any decreasing rational sequence $\{1/q_i\}_{i=1}^\infty$ with $q_1=2$, there exists a neighborhood of the circle, such that there is no other billiard domain preserving $\{1/q_i\}_{i=1}^\infty$ caustics.}\\

The former question can be generalized to the case with finitely many caustics preserved; As for the latter one, it can be considered as a generalization of the Theorem in \cite{ADK}.
\section{Appendix }
\subsection{the calculation of $D^{(3)}(\theta)$} This part can be embedded into (\ref{d3-un}):
\be
D^{(3)}(\theta)&=&\sum_{k=0}^{2}-\frac {{\sqrt3}}{2}{\theta_k^1}^2+{\sqrt3}n'(\theta_k^0)\theta_k^1+\frac{1}{8\sqrt3}\Big{[}n(\theta_{{k}}^0)-n(\theta_{{k+1}}^0)+\sqrt3\theta_{k+1}^1+\sqrt3\theta_k^1\Big{]}^2\nonumber\\
&=&-\frac{\sqrt3}{2}\Big{(}{\theta_1^1}^2+{\theta_2^1}^2\Big{)}+\sqrt3\Big{(}n'(\theta_1^0)\theta_1^1+n'(\theta_2^0)\theta_2^1\Big{)}+\frac {1}{8\sqrt3}\Big{[}\nonumber
\ee
\be
& &(n(\theta_0^0)-n(\theta_1^0)+\sqrt3\theta_1^1)^2+(n(\theta_1^0)-n(\theta_2^0)+\sqrt3\theta_2^1+\sqrt3\theta_1^1)^2\nonumber\\
& &+(n(\theta_2^0)-n(\theta_0^0)+\sqrt3\theta_2^1)^2\Big{]}\nonumber
\ee
\be
&=&-\frac {\sqrt3}{2\cdot169}\Bigg\{(n_0,n_1,n_2)\cdot
\begin{pmatrix}
    3  &1&-4    \\
     1  &\frac{1}{3}&-\frac{4}{3}    \\
      -4  &-\frac{4}{3}&\frac{16}{3}    \\  
\end{pmatrix}\cdot
\begin{pmatrix}
   n_0 \\
     n_1\\
     n_2 
\end{pmatrix}+2(n_0,n_1,n_2)\cdot\nonumber\\
& &\begin{pmatrix}
   0   &16\sqrt3&4\sqrt3    \\
   0   &\frac{16}{\sqrt3} &\frac{4}{\sqrt3}    \\
   0   &-\frac{64}{\sqrt3} &-\frac{16}{\sqrt3}    \\ 
\end{pmatrix}\cdot
\begin{pmatrix}
     n'_0   \\
     n'_1\\
     n'_2  
\end{pmatrix}+(n'_0,n'_1,n'_2)\cdot\begin{pmatrix}
    0  & 0&0   \\
    0  & 256&64\\
    0 & 64 & 16
\end{pmatrix}\cdot
\begin{pmatrix}
      n'_0    \\
      n'_1\\
     n'_2  
\end{pmatrix}\nonumber\\
& &+\nonumber(n_0,n_1,n_2)\cdot
\begin{pmatrix}
    \frac 1{3}  &-\frac{4}{3}&1    \\
      -\frac 4{3}  &\frac{16}{3}&-4    \\
      1  &-4&3    \\  
\end{pmatrix}\cdot
\begin{pmatrix}
   n_0 \\
     n_1\\
     n_2 
\end{pmatrix}+2(n_0,n_1,n_2)\cdot\nonumber\\
& &\begin{pmatrix}
   0   &-4\sqrt3 &-\frac{16}{\sqrt3}    \\
   0   &16\sqrt3 &\frac{64}{\sqrt3}    \\
   0   &-12\sqrt3 &-16\sqrt3    \\ 
\end{pmatrix}\cdot
\begin{pmatrix}
     n'_0   \\
     n'_1\\
     n'_2  
\end{pmatrix}+(n'_0,n'_1,n'_2)\cdot\begin{pmatrix}
    0  & 0&0   \\
    0  &144&192\\
    0 & 192& 256
\end{pmatrix}\cdot
\begin{pmatrix}
      n'_0    \\
      n'_1\\
     n'_2  
\end{pmatrix}\Bigg\}\nonumber
\ee
\be
& &+\frac{\sqrt3}{13}\Bigg\{(n'_0,n'_1,n'_2)\cdot\begin{pmatrix}
    0  &0&0    \\
     0 &16&0\\
     0 &4&0  
\end{pmatrix}\cdot\begin{pmatrix}
      n'_0   \\
      n'_1\\ 
      n'_2
\end{pmatrix}+(n_0,n_1,n_2)\cdot\begin{pmatrix}
   0   & \sqrt3&0   \\
    0  & \frac{1}{\sqrt3}&0\\
    0 &-\frac{4}{\sqrt3}&0 
\end{pmatrix}\cdot\nonumber\\
& &\begin{pmatrix}
      n'_0    \\
      n'_1\\
      n'_2 
\end{pmatrix}+(n'_0,n'_1,n'_2)\cdot\begin{pmatrix}
    0  &0&0    \\
     0 &0&12\\
     0 &0&16 
\end{pmatrix}\cdot\begin{pmatrix}
      n'_0   \\
      n'_1\\ 
      n'_2
\end{pmatrix}+(n_0,n_1,n_2)\cdot\begin{pmatrix}
   0   & 0&-\frac{1}{\sqrt3}   \\
    0  & 0&\frac{4}{\sqrt3}\\
    0 &0&-\sqrt3
\end{pmatrix}\cdot\nonumber\\
& &\begin{pmatrix}
      n'_0   \\
      n'_1\\
      n'_2
\end{pmatrix}\Bigg\}+\frac{1}{8\sqrt3}\Bigg\{\Big{[}(n_0,n_1,n_2)\begin{pmatrix}
     \frac{16}{13}    \\
      -\frac{12}{13}\\
     -\frac{4}{13}  
\end{pmatrix}+(n'_0,n'_1,n'_2)\begin{pmatrix}
      0    \\
      \frac{16\sqrt3}{13}\\
      \frac{4\sqrt3}{13}  
\end{pmatrix}\Big{]}^2+\nonumber\\
& &\Big{[}(n_0,n_1,n_2)\begin{pmatrix}
    \frac{2}{13}    \\
      \frac{18}{13}\\
    -\frac{20}{13}  
\end{pmatrix}+(n'_0,n'_1,n'_2)\begin{pmatrix}
      0    \\
      \frac{28\sqrt3}{13}\\
      \frac{20\sqrt3}{13} 
\end{pmatrix}\Big{]}^2+\Big{[}(n_0,n_1,n_2)\begin{pmatrix}
     -\frac{14}{13}    \\
      \frac{4}{13}\\
     \frac{10}{13}  
\end{pmatrix}+\nonumber\\
& &(n'_0,n'_1,n'_2)\begin{pmatrix}
      0    \\
      \frac{12\sqrt3}{13}\\
      \frac{16\sqrt3} {13}
\end{pmatrix}\Big{]}^2\Bigg\}\nonumber
\ee
\be
&=&-\frac{\sqrt3}{338}\Bigg\{\cN\begin{pmatrix}
    \frac{10}{3}  &  -\frac{1}{3}&-3  \\
    -\frac{1}{3}  &  \frac{17}{3}&-\frac{16}{3}\\
    -3&-\frac{16}{3}&\frac{25}{3}
\end{pmatrix}\cN+2\cN\begin{pmatrix}
     0 & \frac{36}{\sqrt3}&- \frac{4}{\sqrt3}  \\
      0&\frac{64}{\sqrt3} &\frac{68}{\sqrt3} \\
      0&-\frac{100}{\sqrt3}&-\frac{64}{\sqrt3}
\end{pmatrix}\cN'+\nonumber\\
& &\cN'\begin{pmatrix}
      0&0&0    \\
    0  &  400&256\\
    0&256&272
\end{pmatrix}\cN'\Bigg\}+\frac{\sqrt3}{13}\Bigg\{\cN'\begin{pmatrix}
     0 & 0&0   \\
      0&  16&12\\
      0&4&16
\end{pmatrix}\cN'+\cN\begin{pmatrix}
     0 & \frac{3}{\sqrt3} & -\frac{1}{\sqrt3} \\
    0  &  \frac{1}{\sqrt3}&\frac{4}{\sqrt3}\\
    0&-\frac{4}{\sqrt3}&-\sqrt3
\end{pmatrix}\cN'\Bigg\}\nonumber\\
& &+\frac{1}{1352\sqrt3}\Bigg\{\cN\begin{pmatrix}
    456 &  -212&-244  \\
    -212  & 484&-272\\
    -244&-272&516
\end{pmatrix}\cN+2\cN\begin{pmatrix}
     0 & 144\sqrt3& -120\sqrt3  \\
      0&360\sqrt3 &376\sqrt3 \\
      0&-504\sqrt3&-256\sqrt3
\end{pmatrix}\cN'+\nonumber\\
& &\cN'\begin{pmatrix}
      0&0&0    \\
    0  &  3552&2448\\
    0&2448&2016
\end{pmatrix}\cN'\Bigg\}\nonumber\\
&=&\frac{1}{13\sqrt3}\Bigg[\cN\begin{pmatrix}
     4 &-2 &-2  \\
     -2& 4 &-2\\
     -2&-2&4
\end{pmatrix}\cN+\cN\begin{pmatrix}
     0 & 3\sqrt3&-3\sqrt3   \\
     0 & 3\sqrt3&6\sqrt3\\
     0&-6\sqrt3&-3\sqrt3
\end{pmatrix}\cN'\nonumber\\
& &+\cN'\begin{pmatrix}
    0  & 0&0   \\
     0 &  36&30\\
     0&6&36
\end{pmatrix}\cN'\Bigg]\nonumber\\
&=&\sum_{k,l}\Big[n_kn_lV^t_{(3)}(k)A_{(3)}V_{(3)}(l)+ln_kn_lV^t_{(3)}(k)B_{(3)}V_{(3)}(l)\nonumber\\
& &+kln_kn_lV^t_{(3)}(k)C_{(3)}V_{(3)}(l)\Big]\exp({\bf i}(k+l)\theta)\hspace{30pt}(\star)\nonumber\\
&=&{\sum_{k,l}n_kn_lc^{(3)}(k,l)}\exp({\bf i}(k+l)\theta)\nonumber
\ee
with the coefficient $c^{(3)}(k,l)$ denoted by 
\be
c^{(3)}(k,l)&=&\frac{1}{13\sqrt 3}\Big[4-2e^{i\frac{2\pi}{3}k}-2e^{i\frac{4\pi}{3}k}+(3\sqrt3-2)e^{i\frac{2\pi}{3}l}+(4+3\sqrt3 l+36kl)e^{i\frac{4\pi}{3}(k+l)}\nonumber\\
& &+(6kl-2-6\sqrt3 l)e^{i\frac{4\pi}{3}k+i\frac{2\pi}{3}l}-(2+3\sqrt3 l)e^{i\frac{4\pi}{3}l}+(-2+6\sqrt3 l+30kl)e^{i\frac{2\pi}{3}k+i\frac{4\pi}{3}l}\nonumber\\
& &+(4-3\sqrt3l+36kl)e^{i\frac{4\pi}{3}(k+l)}\Big],
\ee
 and
 \[
A_{(3)}=\begin{pmatrix}
     4 &-2 &-2  \\
     -2& 4 &-2\\
     -2&-2&4
\end{pmatrix},\ B_{(3)}=\begin{pmatrix}
     0 & 3\sqrt3&-3\sqrt3   \\
     0 & 3\sqrt3&6\sqrt3\\
     0&-6\sqrt3&-3\sqrt3
\end{pmatrix},\ C_{(3)}=\begin{pmatrix}
    0  & 0&0   \\
     0 &  36&30\\
     0&6&36
\end{pmatrix},
\]
\[
V_{(3)}(k)=\begin{pmatrix}
    1     \\
    \exp(i\frac{2k\pi}{3})\\
    \exp(i\frac{4k\pi}{3}) 
\end{pmatrix},
\]
 $\cN=(n(\theta),n(\theta+\frac{2\pi}{3}),n(\theta+\frac{4\pi}{3}))$ and $\cN'=(n'(\theta),n'(\theta+\frac{2\pi}{3}),n'(\theta+\frac{4\pi}{3}))$. Recall that 
 \[
\theta_1^1=(n(\theta_0^0),n(\theta_1^0),n(\theta_2^0))\cdot
\begin{pmatrix}
      -\frac{1}{3\sqrt3} \\
      -\frac{1}{3\sqrt3}\\
      \frac{2}{3\sqrt3} 
\end{pmatrix}+(n'(\theta_0^0),n'(\theta_1^0),n'(\theta_2^0))\cdot
\begin{pmatrix}
    0   \\
      \frac 43 \\
      \frac23  
\end{pmatrix}
\]
and
\[
\theta_2^1=(n(\theta_0^0),n(\theta_1^0),n(\theta_2^0))\cdot
\begin{pmatrix}
      \frac{1}{3\sqrt3} \\
      -\frac{2}{3\sqrt3}\\
      \frac{1}{3\sqrt3} 
\end{pmatrix}+(n'(\theta_0^0),n'(\theta_1^0),n'(\theta_2^0))\cdot
\begin{pmatrix}
    0   \\
      \frac 23 \\
      \frac43  
\end{pmatrix}
\]
are used in the first few steps of previous computation, which are given by Remark \ref{theta-1-est}. \textr{Recall that $k,l\in\cZ_{2,3}$ because $n\in T_2\cap T_3\backslash 0$, step $(\star)$ actually can be specified to the following properties: Notice that
\[
V_{(3)}(6k+1)=\begin{pmatrix}
    1     \\
    \exp(i\frac{2\pi}{3})\\
    \exp(i\frac{4\pi}{3}) 
\end{pmatrix},\quad\forall k\in\Z
\]
and
\[
V_{(3)}(6k-1)=\begin{pmatrix}
    1     \\
    \exp(i\frac{4\pi}{3})\\
    \exp(i\frac{2\pi}{3}) 
\end{pmatrix},\quad\forall k\in\Z,
\]
then for an arbitrary $3\times 3$ matrix 
\[
\Xi=\begin{pmatrix}
   \chi_{11}  & \chi_{12} &\chi_{13}   \\
    \chi_{21}  & \chi_{22} & \chi_{23}\\
    \chi_{31} & \chi_{32} &\chi_{33}
\end{pmatrix},
\]
we have
\[
V^t_{(3)}(6p+1)\cdot \Xi\cdot V_{(3)}(6q+1)=( \chi_{11} + \chi_{23} + \chi_{32} )+( \chi_{21} + \chi_{12} + \chi_{33} )e^{i\frac{2\pi}{3}}+( \chi_{31} + \chi_{22} + \chi_{13} )e^{i\frac{4\pi}{3}},
\]
\[
V^t_{(3)}(6p-1)\cdot \Xi\cdot V_{(3)}(6q-1)=( \chi_{11} + \chi_{23} + \chi_{32} )+( \chi_{21} + \chi_{12} + \chi_{33} )e^{i\frac{4\pi}{3}}+( \chi_{31} + \chi_{22} + \chi_{13} )e^{i\frac{2\pi}{3}},
\]
\[
V^t_{(3)}(6p+1)\cdot \Xi\cdot V_{(3)}(6q-1)=( \chi_{11} + \chi_{22} + \chi_{33} )+( \chi_{21} + \chi_{32} + \chi_{13} )e^{i\frac{2\pi}{3}}+( \chi_{31} + \chi_{12} + \chi_{23} )e^{i\frac{4\pi}{3}},
\]
\[
V^t_{(3)}(6p-1)\cdot \Xi\cdot V_{(3)}(6q+1)=( \chi_{11} + \chi_{22} + \chi_{33} )+( \chi_{21} + \chi_{32} + \chi_{13} )e^{i\frac{4\pi}{3}}+( \chi_{31} + \chi_{12} + \chi_{23} )e^{i\frac{2\pi}{3}}.
\]
Aforementioned equalities reveal the different summations of modulars of the matrix $\Xi$. By taking $\Xi=A_{(3)}+B_{(3)}+C_{(3)}$, we can show that:
}
\[
c^{(3)}(6p+1,6q+1)=0,\quad\forall p,q\in\Z,
\]
\[
c^{(3)}(6p-1,6q-1)=0,\quad\forall p,q\in\Z,
\]
\[
c^{(3)}(6p+1,6q-1)={\frac{1}{13\sqrt3}\Big[18-27l{\bf i}+kl(66+24\exp({\bf i}\frac{4\pi}{3}))\Big]},\quad\forall p,q\in\Z,
\]
and
\[
c^{(3)}(6p-1,6q+1)={\frac{1}{13\sqrt3}\Big[18+27l{\bf i}+kl(66+24\exp({\bf i}\frac{2\pi}{3}))\Big]},\quad\forall p,q\in\Z.
\]
\vspace{20pt}

\noindent\textbf{Acknowledgement} The authors thank Guan Huang for checking the details of calculation, and thank Jacopo de Simoi for proposing some useful suggestions of this topic.

\end{document}